\documentclass[11pt,reqno]{amsart}

\usepackage[numeric,initials,nobysame]{amsrefs}

\usepackage{mathrsfs}
\usepackage{bbm}
\usepackage{mathtools}
\usepackage[margin=0.98in]{geometry}
\usepackage{setspace}
\usepackage{enumerate}

\newcommand{\inv}{\operatorname{inv}}

\title{Mallows permutations as stable matchings}
\date{4 June 2018}
\author{Omer Angel}
\author{Alexander E.\ Holroyd}
\author{Tom Hutchcroft}
\author{Avi Levy}
\address{Omer Angel, Department of Mathematics, University of British Columbia}
\email{angel@math.ubc.ca}
\address{Alexander E.\ Holroyd}
\email{holroyd@uw.edu}
\address{Tom Hutchcroft, Statslab, DPMMS, University of Cambridge}
\email{t.hutchcroft@maths.cam.ac.uk}
\address{Avi Levy, Microsoft Corporation, Redmond, WA 98052}
\email{avi.levy@microsoft.com}

\usepackage{%
  amssymb,amsmath,amsthm,amsrefs,%
   wasysym,textcomp,
  cleveref
}

\renewcommand{\P}{\mathbb P}
\newcommand{\E}{\mathbb E}

\newcommand{\Z}{\mathbb Z}

\newcommand{\cF}{\mathcal F}

\newcommand{\sA}{\mathscr A}

\def\nc{\newcommand}
\nc\m\mathbb
\nc\Mal{\operatorname{Mal}}
\nc\MalPref{\operatorname{MalPref}}
\nc\sgn{\operatorname{sgn}}
\nc\matched{\operatorname{Matched}}
\nc\unmatched{\operatorname{Unmatched}}
\nc\match{\operatorname{Match}}
\nc\PrefStruc{\operatorname{PrefStruc}}

\nc\dto[1]{\xrightarrow[#1]{d}}

\nc\Var{\operatorname{Var}}
\nc\Cov{\operatorname{Cov}}
\nc\lat{\operatorname{Lat}}
\nc\elat{\operatorname{ELat}}
\nc\Res{\operatorname{Res}}
\nc\Fix{\operatorname{Fix}}
\nc\stab{\operatorname{StabMatch}}
\nc\Nbhd{\operatorname{Nbhd}}
\nc\Bnd{\operatorname{Bnd}}
\nc\gr{\operatorname{gr}}
\nc\id{\operatorname{id}}

\newcommand{\flow}{\operatorname{Fl}}
\newcommand{\eps}{\varepsilon}

\nc\malefemale{\mathrel{\ooalign{$\male$\cr\hidewidth\raise-.25ex\hbox{$\female\mkern4.5mu$}\cr}}}

  \theoremstyle{plain}
    \newtheorem{thm}{Theorem}
    \newtheorem*{thm*}{Theorem}
    \newtheorem{prop}[thm]{Proposition}
    \newtheorem{lemma}[thm]{Lemma}
    \newtheorem{cor}[thm]{Corollary}

  \theoremstyle{definition}
    
  \theoremstyle{remark}
    \newtheorem*{remark}{Remark}

\crefname{theorem}{Theorem}{Theorems}
\crefname{thm}{Theorem}{Theorems}
\crefname{lemma}{Lemma}{Lemmas}
\crefname{lem}{Lemma}{Lemmas}
\crefname{remark}{Remark}{Remarks}
\crefname{prop}{Proposition}{Propositions}
\crefname{defn}{Definition}{Definitions}
\crefname{corollary}{Corollary}{Corollaries}
\crefname{cor}{Corollary}{Corollaries}
\crefname{conjecture}{Conjecture}{Conjectures}
\crefname{question}{Question}{Questions}
\crefname{chapter}{Chapter}{Chapters}
\crefname{section}{Section}{Sections}
\crefname{figure}{Figure}{Figures}

\newcommand{\qPoch}{(q)_\infty}

\begin{document}

\begin{abstract}
  We show that the Mallows measure on permutations of $1,\ldots,n$ arises as the law of the unique Gale-Shapley stable matching of the random bipartite graph conditioned to be perfect, where preferences arise from a total ordering of the vertices but are restricted to the (random) edges of the graph.
  We extend this correspondence to infinite intervals, for which the situation is more intricate. We prove that almost surely every stable matching of the random bipartite graph obtained by performing Bernoulli percolation on the complete bipartite graph $K_{\Z,\Z}$ falls into one of two classes: a countable family $(\sigma_n)_{n\in\Z}$ of \emph{tame} stable matchings, in which the length of the longest edge crossing $k$ is $O(\log |k|)$ as $k\to\pm \infty$, and an uncountable family of \emph{wild} stable matchings, in which this length is $\exp \Omega(k)$ as $k\to +\infty$.
  The tame stable matching $\sigma_n$ has the law of the Mallows permutation of $\Z$ (as constructed by Gnedin and Olshanski) composed with the shift $k\mapsto k+n$.
  The permutation $\sigma_{n+1}$ dominates $\sigma_{n}$ pointwise, and the two permutations are related by a shift along a random strictly increasing sequence.
\end{abstract}

\maketitle

\section{Introduction}

In this paper we will establish a connection between two classical objects: the Mallows measure on permutations and Gale-Shapley stable marriage.
The \textbf{Mallows measure} $\Mal^n_q$ on permutations of $\{1,\ldots,n\}$ with parameter ${q\in[0,1]}$ is the probability measure that assigns to each permutation $\sigma\in S_n$ a probability proportional to $q^{\inv(\sigma)}$, where
  $\inv(\sigma)$ is the \textbf{inversion number} of $\sigma$, given by
  $$
    \inv(\sigma)=\#\bigl\{(i,j)\in \{1,\ldots,n\}^2\colon i<j\text{ but }\sigma(i)>\sigma(j)\bigr\}.
  $$
  More generally, we define the Mallows measure $\Mal^I_q$ on permutations of a general finite interval $I \subseteq \Z$ by shifting the index.
  The Mallows measure was extended to permutations of infinite intervals by Gnedin and Olshanski \cite{gnedin2010q,gnedin2012two}, who showed that for $q\in [0,1)$ and an infinite interval $I \subseteq \Z$, the measures $\Mal^{I\cap[-n,n]}_q$ converge weakly (with respect to the topology of pointwise  convergence) to a probability measure $\Mal_q^I$ on permutations of $I$.  We call this limit the Mallows measure on permutations of $I$ with parameter $q$. They also characterised the Mallows permutation of $\Z$, together with its compositions with shifts, as the unique random permutations of $\Z$ with a property that they called \emph{$q$-exchangeability}, which is equivalent to being a Gibbs measure on permutations of $\Z$ with respect to the Hamiltonian $H(\sigma)= \inv(\sigma)$ and inverse temperature $\beta=-\log q$.

  The Mallows measure was originally introduced in the context of  statistical ranking theory~\cite{mallows}. It has recently enjoyed substantial interest among both pure and applied mathematicians. In particular, analysis has been carried out of the cycle structure \cite{gladkich2016cycle} and the longest increasing subsequence \cite{basu2016limit,MR3334280,mueller2013length} of a Mallows permutation, of the longest common subsequence of two independent Mallows permutations \cite{1611.03840}, and of mixing times of related Markov chains \cite{benjamini2005mixing,diaconis2000analysis}.
  The Mallows permutation has also been studied as a statistical physics model \cite{starr2009thermodynamic,starr2015phase}, and has found applications in learning theory \cite{braverman2009sorting} and in the theory of finitely dependent processes~\cite{1706.09526}. The Mallows measure also arises as a stationary measure of the asymmetric exclusion process (ASEP) \cite{MR2857238}.

  In this paper, we show that, for both finite and infinite intervals, the Mallows permutation arises as a \emph{stable matching} of the random bipartite graph on the interval.
  In particular, we obtain a new construction of the Mallows permutation of an infinite interval.
  The finite case follows in a straightforward way by consideration of known algorithms for sampling from the Mallows distribution, while the infinite case is more subtle and requires a more delicate treatment.

The notion of stable matching was introduced in the hugely influential work of Gale and Shapley~\cite{gale1962college}.
Since then, thousands of articles on the topic have been written, and Nobel Memorial Prizes in Economics have been awarded to Roth and Shapley for related work.

Let us now describe informally the random stable matchings with which we shall be concerned. Suppose that we have a set of males and a set of females who seek to be matched into heterosexual pairs, in accordance with preferences defined as follows. %
Both the set of males and the set of females are ranked according to a universally agreed order of attractiveness.
However, each male-female pair has a probability ${q\in[0,1)}$ of being incompatible, independently of all other pairs, meaning that neither will consider the other as a partner  under any circumstances.
Attractiveness and compatibility are the only factors affecting preferences.
In particular, if each set is finite, a female's first choice for partner is the most attractive male she is compatible with, her second choice is the second most attractive male she is compatible with, and so on. A matching is a collection of compatible pairs such that each individual is in at most one pair.
A matching is stable if there does not exist a compatible male-female pair who would both prefer to be matched to each other over their current status, where any compatible partner is preferred to being unmatched.

We now introduce notation in order to make these definitions more formal. A \textbf{matching} of a graph $G$  is a set of edges no two of which share a vertex. A matching is \textbf{perfect} if every vertex has a partner.
 Given intervals $I,J\subseteq \Z$, we write $K_{I,J}$ for the graph whose vertex set is $(I \times \{\male\})\cup (J\times \{\female\})$ and whose edge set is $\bigl\{ \{(i,\male),(j,\female) \} : i \in I,\ j\in J\bigr\}$.
%
%
%
If $G$ is a subgraph of $K_{I,J}$, we say that $(i,\male)$ and $(j,\female)$ are \textbf{compatible} if there is an edge between them in $G$.
Thus, a matching of $K_{I,J}$ is a matching of $G$ if and only if every matched pair is compatible.
We identify each matching of $K_{I,J}$ with a function $\sigma: I \to J \cup \{-\infty\}$ by setting $\sigma(i)=j$ if $(i,\male)$ is matched to $(j,\female)$ and setting $\sigma(i)=-\infty$ if $(i,\male)$ is unmatched.
The function $\sigma^{-1}:J \to I \cup \{-\infty\}$ is defined similarly by interchanging the roles of $\male$ and $\female$.
In particular, if $I=J$ and the matching is perfect then $\sigma$ is a permutation of $I$ and $\sigma^{-1}$ is its inverse; this yields a bijection between perfect matchings of $K_{I,I}$ and permutations of $I$.
A matching of a subgraph $G$ of $K_{I,J}$ is \textbf{stable} if there does not exist a pair $\{(i,\male),(j,\female)\}$ such that $(i,\male)$ is compatible with $(j,\female)$, $\sigma(i) < j$, and $\sigma^{-1}(j)<i$.

Given $p\in [0,1]$, we let $K_{I,I}(p)$ be the subgraph of $K_{I,I}$ with the same vertex set as $K_{I,I}$ and where each edge is included independently at random with probability $p$.

\begin{figure}
{}\hfill\includegraphics[width=.28\textwidth]{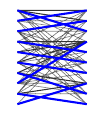}
\hfill\includegraphics[width=.28\textwidth]{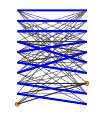}
\hfill{}
\caption{Two realizations of the random bipartite graph $K_{10,10}(0.6)$, with more attractive individuals at the top, together with their unique stable matchings, shown by thick blue lines.
The matching is perfect in the left figure, but not in the right figure.  Orange discs indicate unmatched individuals.
Thin grey (solid and dotted) lines represent edges connecting incompatible pairs.
Edges between compatible individuals that are not in the matching are not shown.
The matching can be computed without examining all edges: those that need to be examined are the edges of the matching (whose endpoints are compatible), and the solid grey edges (which are all those edges whose endpoints must be incompatible in order for the matching to be stable).  In the case of a perfect matching, these solid grey edges are simply all those that connect the upper endpoints of a crossing pair of edges of the matching.}
   \label{fig:matching}
\vspace{-0.75em}
\end{figure}

\begin{prop}\label{prop:finite}
  Let $p\in (0,1]$, let $q=1-p$, and let $I \subseteq \Z$ be an interval that is bounded above. Then the following hold.
  \begin{enumerate}
  \item[\emph{i.}] Every subgraph of $K_{I,I}$ has a unique stable matching, so that in particular $K_{I,I}(p)$ has a unique stable matching almost surely.
  \item[\emph{ii.}] If $I$ is finite then the unique stable matching of the random subgraph $K_{I,I}(p)$ is perfect with probability $\prod_{k=1}^{|I|} (1-q^k)$.
  \item[\emph{iii.}] If $I$ is infinite then the unique stable matching of $K_{I,I}(p)$ is perfect almost surely.
  \item[\emph{iv.}] Conditional on the event that the unique stable matching of $K_{I,I}(p)$ is perfect, it is distributed as a Mallows permutation of $I$ with parameter $q$.
  \end{enumerate}
\end{prop}

We remark that the limit $(q)_\infty:=\prod_{k=1}^{\infty}(1-q^k)$ as $|I|\to\infty$ of the probabilities appearing in Proposition \ref{prop:finite}.ii is positive
 but strictly less than $1$ 
 for each $q\in (0,1)$.
We also remark, as a point of general interest, that the asymptotics of this infinite product were computed by Hardy and Ramanujan \cite{MR1575586} to be
\begin{align*}
  \qPoch := \prod_{k=1}^{\infty}(1-q^k) \sim
   \sqrt{ 2 \pi (1-q) } \exp \left[ - \frac{\pi^2}{ 6 (1-q)} \right]
   \qquad \text{ as }q\nearrow1.
\end{align*}
The function $\qPoch$ is the reciprocal of the generating function of integer partitions. It is also known as both the \textbf{$q$-Pochhammer symbol}  and the \textbf{Euler function}, owing to its role in Euler's pentagonal number theorem \cite{EulerPentagonal}.


\begin{figure}[t]
{}\hfill\includegraphics[width=.18\textwidth]{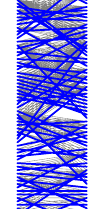}
\hfill\includegraphics[width=.18\textwidth]{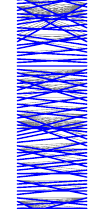}
\hfill\includegraphics[width=.18\textwidth]{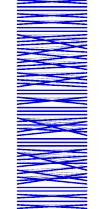}
\hfill {}
\caption{A portion of the balanced tame stable matching of $K_{\Z,\Z}(p)$, for $p=0.2$ (left), $p=0.4$ (middle), and $p=0.6$ (right), shown by blue lines. The law of the matching corresponds to the Mallows measure on permutations of $\Z$. Thin grey lines indicate edges whose endpoints must be incompatible  for this matching to be stable. Other edges are omitted.}
\vspace{-0.7em}
\label{fig:tame}
\end{figure}

The situation for intervals that are unbounded from above (so that there do not exist maximally attractive individuals) in very different, and is the main topic of this paper.
Indeed, for $p\in (0,1)$ the random graph $K_{\m Z,\m Z}(p)$ has \emph{uncountably many} stable matchings, and even uncountably many stable matchings that are not perfect.
We will prove, however, that the stable matchings of $K_{\Z,\Z}(p)$ fall into two sharply distinguished classes:
a countable family of \textbf{tame} matchings which correspond to compositions of the Mallows permutation of $\Z$ with shifts, and an uncountable family of \textbf{wild} stable matchings. Moreover, the tame and wild stable matchings have quantitatively very different behaviours.

To state these results we introduce some more definitions.
For an interval $I \subseteq \Z$, a matching $\sigma$ of $K_{I,I}$ and $i\in I$, we define the quantities
\begin{align*}
  L_+\bigl(\sigma,i+\tfrac{1}{2}\bigr)&=
  \#\Bigl\{j\leq i+\tfrac{1}{2} \text{ such that } \sigma(j)\geq i+ \tfrac{1}{2} \Bigr\}
\intertext{and}
  L_-\bigl(\sigma,i+\tfrac{1}{2}\bigr)&=
  \#\Bigl\{j\geq i+\tfrac{1}{2} \text{ such that } \sigma(j)\leq i+ \tfrac{1}{2} \Bigr\}.
\end{align*}
That is, $L_+(\sigma,i+\tfrac{1}{2})$ and $L_-(\sigma,i+\tfrac{1}{2})$ are the numbers of edges crossing over $i+\tfrac{1}{2}$ in each direction.
We say that the matching $\sigma$ is \textbf{locally finite} if $L_+(\sigma,i+\tfrac{1}{2}) + L_-(\sigma,i+\tfrac{1}{2})$ is finite for some (and hence every) $i\in I$. Note that matchings of intervals other than $\Z$ are always locally finite.
If $\sigma$ is perfect and locally finite, we define the \textbf{flow} of $\sigma$ to be
$$
\flow(\sigma)= L_+\left(\sigma,i+\tfrac{1}{2}\right)-L_-\left(\sigma,i+\tfrac{1}{2}\right),
$$
which is easily seen to be independent of $i$.
We say that $\sigma$ is \textbf{balanced} if it is perfect, locally finite, and has flow zero.

For each matching $\sigma$ of $K_{\Z,\Z}$ and each $i \in \Z$, we also define
\begin{equation*}
  M(\sigma,i+\tfrac{1}{2}) = \max\Big\{|\sigma(j)-j| : j < i {+} \tfrac{1}{2} < \sigma(j) \text{ or } -\infty < \sigma(j) < i {+} \tfrac{1}{2} < j \Big\}
\end{equation*}
to be the length of the longest edge in the matching crossing $i+\tfrac{1}{2}$, where the maximum of the empty set is taken to be zero.
We say that a matching $\sigma$ of $K_{\Z,\Z}$ is \textbf{tame} if it is  locally finite and
\[
  \limsup_{i\to \pm \infty}\frac{M\bigl(\sigma,i+\tfrac{1}{2}\bigl)}{\log |i|} < \infty,
\]
that is, if the maximum length of an edge  of $\sigma$ crossing $i+\tfrac{1}{2}$ is at most logarithmically large in $i$.
On the other hand, we say that a  matching  $\sigma$ is \textbf{wild}
if
\[
  \liminf_{i\to +\infty} \frac{\log M\bigl(\sigma,i+\tfrac{1}{2}\bigl)}{i} >0
\]
that is, if the maximum length of an edge crossing $i+\tfrac{1}{2}$ grows at least exponentially as $i \to+ \infty$. In particular, every matching that is not locally finite is wild.
There is a wide gulf between tameness and wildness -- it is easy to construct matchings of $K_{\Z,\Z}$ that are neither tame nor wild.  However, our stable matchings are either tame or wild.

\begin{thm}[Tame/wild dichotomy]
  \label{thm:tameorwild}
  Let $p\in (0,1]$, let $q=1-p$, and consider the random bipartite graph $K_{\Z,\Z}(p)$. Almost surely, every locally finite stable matching of $K_{\Z,\Z}(p)$ is perfect, and every stable matching of $K_{\Z,\Z}(p)$ is either tame or wild.
\end{thm}

Simulated examples of tame and wild matchings of $K_{\Z,\Z}(p)$ are depicted in \cref{fig:tame,fig:flow,fig:wild}. Note that the definition of wildness is asymmetric, and does not say anything about the behaviour as $i\to-\infty$. In fact, we will show that for every perfect, locally finite stable matching $\sigma$ of $K_{\Z,\Z}(p)$, there is a tame stable matching that agrees with $\sigma$ at all sufficiently large negative $i$ -- see \cref{cor:eventualcoupling}. Next, we relate the tame matchings of $K_{\Z,\Z}(p)$ to the Mallows permutation, and describe some of their properties.

\begin{figure}[t]
\includegraphics[width=.22\textwidth]{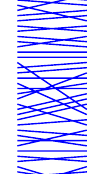}
\hfill\includegraphics[width=.22\textwidth]{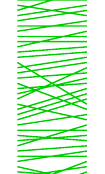}
\hfill\includegraphics[width=.22\textwidth]{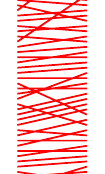}
\hfill\includegraphics[width=.22\textwidth]{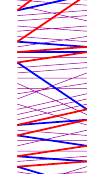}
\caption{Far left: the balanced tame matching $\sigma_0$ of $K_{\Z,\Z}(0.4)$.  Middle left: a matching equal in law to the matching $\sigma_1$ of flow $1$ is obtained by composing $\sigma_0$ with a shift by $1$ (depicted as an upward shift of vertices on the right).  Middle right: the matching $\sigma_1$ for the same realization of the graph $K_{\Z,\Z}(0.4)$.  Far right: the two matchings $\sigma_0$ and $\sigma_1$ superimposed.  Edges unique to $\sigma_0$ and $\sigma_1$ are shown by thick blue and red lines respectively, while edges common to both matchings are shown by thin purple lines.  The symmetric difference consists of a single bi-infinite path that is increasing on both sides.}
\label{fig:flow}
\end{figure}

\begin{thm}[Classification of tame matchings]
  \label{thm:tamechar}
  Let $p\in (0,1]$, let $q=1-p$, and consider the random bipartite graph $K_{\Z,\Z}(p)$.
  Almost surely, the tame stable matchings of $K_{\Z,\Z}(p)$ form a countable family $(\sigma_n)_{n\in \Z}$ with the following properties.
  \begin{enumerate}
  \item[\emph{i.}]
    For each $n\in \Z$, the matching $\sigma_n$ is perfect and has flow $n$.
  \item[\emph{ii.}] The  stable matching $\sigma_n$ is the almost sure pointwise limit of the unique stable matching $\sigma_{n,m}$ of $K_{(-\infty,m],(-\infty,m+n]}(p)$ as $m\to\infty$.
  \item[\emph{iii.}] The matching $\sigma_n$ is distributed as the composition $\sigma$ with the shift $i \mapsto i+n$, where $\sigma$ is a Mallows-distributed permutation of $\m Z$ with parameter $q$.
  \item[\emph{iv.}] If $n' \leq n$, then the males do no better in $\sigma_{n'}$ than in $\sigma_{n}$, and the females do no worse. That is, $\sigma_{n'}(i) \leq \sigma_n(i)$ and $\sigma_{n'}^{-1} \geq \sigma_n^{-1}(i)$ for every $i \in \Z$. Moreover, for every $n\in \Z$, the matchings $\sigma_n$ and $\sigma_{n+1}$ agree except on a strictly increasing sequence $(i_{n,k})_{k\in \Z}$, for which $\sigma_n(i_{n,k})=\sigma_{n+1}(i_{n,k+1})$ for all $k\in \Z$.
\end{enumerate}
\end{thm}

In particular, note that while $\sigma_{n+1}$ is \emph{distributed} as the shift of $\sigma_n$, it is not \emph{equal} to the shift of $\sigma_n$ (see \cref{fig:flow}).
\cref{thm:tamechar}.iii is reminiscent of the situation for stable matchings with general preferences on finite sets, where there are two extremal matchings, one of which is female-optimal and male-pessimal while the other is male-optimal and female-pessimal \cite{gale1962college}.

Finally, we prove that wild stable matchings do indeed exist. (The fact that tame matchings also exist is part of \cref{thm:tamechar}).

\begin{thm}[Existence of wild matchings]
  \label{thm:wildexistence}
  Let $p\in(0,1)$, let $q=1-p$, and consider the random bipartite graph $K_{\Z,\Z}(p)$.
  Almost surely, there exist stable matchings of $K_{\Z,\Z}(p)$ in each of the following categories:
  \begin{itemize}
  \item not perfect,
  \item perfect but not locally finite,
  \item perfect and locally finite, but wild.
  \end{itemize}
  Indeed, there almost surely exist uncountably many stable matchings in each category.
  Also, if $I$ is an infinite interval that is bounded from below, then $K_{I,I}(p)$ has uncountably many wild stable matchings almost surely.
\end{thm}

\begin{figure}
{}\hfill\includegraphics[width=.2\textwidth]{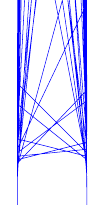}
\hfill\includegraphics[width=.2\textwidth]{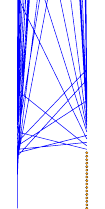}
\hfill\includegraphics[width=.2\textwidth]{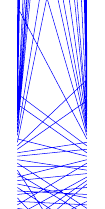}
\hfill {}
\caption{Portions of wild matchings of $K_{\Z,\Z}(0.1)$.  Left: a perfect matching that is not locally finite.  Middle: a non-perfect matching.  Right: a perfect, balanced wild matching.}
\label{fig:wild}
\end{figure}

Matchings satisfying the conditions of \cref{thm:wildexistence} can be constructed via a simple and explicit algorithm. (In particular, no appeal to the axiom of choice is required.)

We remark that \cref{thm:tameorwild} and \cref{thm:wildexistence} are sharp in the sense that, disregarding constants, the definitions of tame and wild cannot be strengthened without the theorem becoming false. In other words, the tame stable matchings of $K_{\Z,\Z}(p)$ have logarithmically long edges, and there exist wild stable perfect matchings of $K_{\Z,\Z}(p)$ in which the longest edge crossing $i+\tfrac{1}{2}$ is at most exponentially large in $i$. See \cref{prop:tamesharp,prop:wildsharp} respectively for the precise statements.

\medskip

\paragraph{\textbf{About the proofs}}
Recall from \cref{prop:finite} that there is a unique stable matching of every interval that is bounded from above.
A central step in the proofs of \cref{thm:tameorwild,thm:tamechar} is to prove that the unique stable matchings of two different intervals, both bounded from above, coincide for all sufficiently large negative $i$, and moreover that the distance elapsed before they couple in this way has an exponential tail.
In particular, we will prove that if $\sigma$ and $\sigma'$ are the unique stable matchings of $K_{(-\infty,0],(-\infty,0]}(p)$ and
$K_{(-\infty,n],(-\infty,n]}(p)$ for some $n \geq 0$, then
\begin{equation}
\label{eq:couplingintro}
\P\Bigl(\sigma(-i) = \sigma'(-i) \text{ for all $i \geq k$}
\Bigr) \geq 1- \left[1-(1-q)\qPoch^2\right]^k.
\end{equation}
Note that the bound on the right of \eqref{eq:couplingintro} does not depend on $n$.
The exact statement required for the proofs of \cref{thm:tameorwild,thm:tamechar} is a little more general than this and is given in \cref{prop:couplingbound}.

To prove this proposition, we consider the \emph{mutual cuts} of $\sigma$ and $\sigma'$.
 If $\sigma$ is a permutation of $\Z$, we say that $i + \tfrac{1}{2}$ is a \textbf{cut} for $\sigma$ if $\sigma$ fixes the sets $\{j \in \Z : j < i +\tfrac{1}{2}\}$ and $\{j \in \Z : j > i +\tfrac{1}{2}\}$. If, in the setting above, $i \geq 0$ and $-i+\tfrac{1}{2}$ is a cut for both $\sigma$ and $\sigma'$, then it follows from \cref{prop:finite} that $\sigma(-j)=\sigma'(-j)$ for all $j\geq i$. Thus, to prove \eqref{eq:couplingintro}, it suffices to prove that
 \begin{equation}
 \label{eq:introcouplingBernoulli}
\Bigl(\mathbbm{1}\!\big[-i +\tfrac{1}{2}  \text{ is a cut for both $\sigma$ and $\sigma'$}\big]\Bigr)_{i \geq 0}
 \end{equation}
 stochastically dominates an i.i.d.\ Bernoulli process with parameter $(1-q)\qPoch^2$.

 The proof of \cref{prop:couplingbound} also yields the following variation of this result concerning a single permutation, which is of independent interest.

\begin{prop}
\label{prop:cutbernoulli}
Let $q\in [0,1)$, let $I$ be an infinite interval, and let $\sigma$ be a random permutation of $I$ drawn from the Mallows distribution with parameter $q$. Then
the process
\[\Bigl(\mathbbm{1}\!\big[i +\tfrac{1}{2}  \text{ is a cut for $\sigma$}\big]\Bigr)_{i \in \Z}\] stochastically dominates an i.i.d.\ Bernoulli process with parameter $(1-q)\qPoch$.
\end{prop}

We note that if $\sigma$ is a Mallows permutation of $\Z$ then
\[
\P(i +\tfrac{1}{2} \text{ is a cut for $\sigma$}) = \qPoch
\]
for every $i \in \Z$; this is an immediate consequence of \cref{prop:finite} and \cref{thm:tamechar}, and is also an easy consequence of the $q$-shuffling algorithm for sampling the Mallows permutation \cite{gnedin2012two}. Thus, the density of the cuts of $\sigma$ and of the Bernoulli process that \cref{prop:cutbernoulli} states that they dominate differ by a factor of $1-q$.

\begin{figure}[t!]
\includegraphics[width=.24\textwidth]{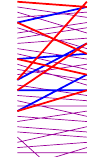}
\caption{Two matchings of equal flow agree at all sufficiently negative locations.  Here $p=0.3$, the two matchings have cuts at $-\tfrac12$ and $-\tfrac32$ respectively, and only negative integer locations are shown.  Edges unique to one or other matching are shown by thick blue and red lines respectively, while edges common to both matchings are shown by thin purple lines.  The symmetric difference (restricted to negative locations) consists of a single path with a unique locally minimal edge.
}
\vspace{-0.3cm}
\end{figure}

We now briefly discuss the proofs of \cref{prop:couplingbound} and \cref{prop:cutbernoulli}.
\cref{prop:finite} naturally leads to several algorithms for sequentially sampling the Mallows permutation, depending on the order in which we choose to reveal the status of the edges in $K_{I,I}(p)$. Different algorithms lend themselves to studying different aspects of the  permutation. For example, the \emph{diagonal exposure} algorithm of Gladkich and Peled \cite{gladkich2016cycle}, which is well suited to studying the cycle structure of the Mallows permutation, is of this form.
To prove \cref{prop:cutbernoulli}, we introduce a new algorithm for sequentially sampling the Mallows permutation that is well suited to studying cuts, and has a natural interpretation in terms of the matching. In this algorithm, an ``alpha'' male prevents less attractive males from finding partners until he himself finds one (at which point another male takes over as the alpha male).  A similar algorithm is implicit in the proof of \cite[Proposition 8.1]{MR3164758}.
 To prove \cref{prop:couplingbound}, we use a variation on this algorithm in which the two matchings $\sigma$ and $\sigma'$ are computed simultaneously.

%

\section{Intervals with maximally attractive individuals}

We now prove \cref{prop:finite}.  As advertised in the introduction, we will also obtain a new proof that the weak limit used to define the Mallows permutation of the infinite interval $(-\infty,0]$ exists.

\begin{proof}[Proof of \cref{prop:finite}]
  Let $I$ be an interval that is bounded from above.
  We may assume without loss of generality that $\max I = 0$.
  Fix a subgraph $G$ of $K_{I,I}$.
  In any stable matching of $G$, the most attractive male must be matched to the most attractive female that is compatible with him.
  Inductively, the $i^{th}$ most attractive male must be matched to the most attractive female he is compatible with among those who are not matched to a more attractive male.
  This shows that the stable matching is unique, and gives an algorithm to compute it.
  Formally, we set $\max\emptyset = -\infty$ and define $\sigma\colon I \to I \cup \{-\infty\}$ recursively by setting
  \[
  \sigma(0) = \max\{k\in I \colon (k,\female)\text{ is compatible with } (0,\male)\}
  \]
  and, for all $1 \leq i < |I|$,
  \begin{equation}
    \label{eq:sigmaalgorithm}
    \sigma(-i) = \max\Bigl(\{k \in I \colon (k,\female)\text{ is compatible with } (-i,\male)\} \setminus \sigma([1-i,0]) \Bigr),
  \end{equation}
  where we use the notation $\sigma(A) = \{\sigma(a) \colon a\in A\}$.
  It follows by induction on $i$ that $\sigma$ is the unique stable matching of $G$.


  Now suppose that
  $G=K_{I,I}(p)$. Then the probability that $\sigma(-i) \not= -\infty$ given $\sigma(0),\ldots,\sigma(-i+1)$ is equal to the probability that $(-i,\male)$ is compatible with $(k,\female)$ for some $k$ in the set
    \begin{equation*}
      A_i = I \setminus \sigma([1-i,0]).
    \end{equation*}
    If $I$ is infinite, then $A_i$ is infinite for every $i$, and we deduce that $\sigma(-i)\neq-\infty$ for every $i\geq0$ almost surely.
    Otherwise, $|I|=n$ for some $n\geq 0$.
    In this case, on the event that $\sigma(-j)\not=-\infty$ for all $0 \leq j<i$, the set $A_i$ has cardinality $n-i$. It follows that
    $$
      \P\bigl(\sigma(-i)\not=-\infty\mid \sigma(0)\not=-\infty,\ldots,\sigma(-i+1)\not=-\infty\bigr)=1-q^{n-i}.
    $$
    Thus, the probability that the unique stable matching is perfect is given by
    $$
    \P\bigl(\sigma(0)\not=-\infty,\ldots,\sigma(-n+1)\not=-\infty\bigr)
    = \prod_{i=0}^{n-1}(1-q^{n-i})
    = \prod_{i=1}^{n}(1-q^{i}).
    $$

  We next show that if $I$ is finite, then the conditional distribution of the unique stable matching $\sigma$ of $K_{I,I}(p)$ given that it is perfect is equal to $\Mal_q^I$.
  Fix a permutation $\tau$ of $I$. We wish to show that the probability that $\sigma=\tau$ is proportional to $q^{\inv(\tau)}$.
  By the recursive formula for $\sigma$ given above, we have that for all $1\leq i\leq n$,
  $$
  \P\bigl(\sigma(-i)=\tau(-i)\mid \sigma(0)=\tau(0),\ldots,\sigma(-i+1)=\tau(-i+1)\bigr)
  = (1-q) q^{\#\{j<-\tau(-i)\colon -\tau^{-1}(-j)>i\}}.
  $$
  Taking the product of these conditional probabilities and observing that
  $$
  \sum_{j=1}^n\#\{j<-\tau(-i)\colon -\tau^{-1}(-j)>i\}=\inv(\tau)
  $$
  yields that $\mathbb P(\sigma=\tau)=q^{\inv(\tau)}(1-q)^n$ as required. Note that this yields a proof of the well-known formula
  \[
    \sum_{\tau \in S_n} q^{\inv(\tau)} = \frac{\prod_{i=1}^n(1-q^i)}{(1-q)^n}.
  \]

  Now suppose that $I=(-\infty,0]$, and let $\sigma$ be the unique stable matching of $K_{I,I}(p)$, which is almost surely perfect.
  It remains to prove that $\sigma$ is the Mallows permutation of $I$ as defined by Gnedin and Olshanki \cite{gnedin2010q,gnedin2012two}.
  That is, we must prove that the law of $\sigma$ is equal to the weak limit of the Mallows measures on permutations of $I_n = [-n,0]$ as $n\to\infty$, i.e., that
  \begin{equation}
    \label{eq:MallowsdefI}
    \P\big(\sigma(i)=x_i \; \forall\, 1 \leq i \leq k\big)
    = \lim_{n\to\infty} \Mal_q^{I_n}(\sigma(i)=x_i\; \forall\, 1 \leq i \leq k)
  \end{equation}
  for every $k\geq 1$ and $x_1,\ldots,x_k\in I$.
  In fact, we will obtain as a corollary a new proof that this weak limit exists, recovering the result of  \cite{gnedin2010q}.
  For each $n\geq 0$, let $\sigma_n$ be the unique stable matching of $K_{I_n,I_n}(p)$ and observe that, by the above algorithm,
  \[
    \sigma_n(-i)= \begin{cases}
      \sigma(-i) & \text{ if } \sigma(-i) \geq -n\\
      -\infty & \text{ if } \sigma(-i) < -n.
    \end{cases}
  \]
  Observe that for every $n \geq 1$, every $k \leq n$ and every $x_1,\ldots,x_k \in I$, we have that, by a similar analysis to above,
  \begin{align*}
    \P\big(\sigma_n \text{ perfect} \mid \sigma(-i) =x_i \hspace{0.5em} &\forall\, 0\leq i\leq k\big)\\ &=
    \P\big(\sigma_n(-j) \neq -\infty\hspace{0.5em} \forall\, 0 \leq j \leq n \mid \sigma(-i) =x_i \hspace{0.5em} \forall\, 0\leq i \leq k\big)\\
    &=
    \mathbbm{1}\!\bigl[ x_i \in [-n,0] \hspace{0.5em} \forall\, 1\leq i \leq k \bigr] \prod_{i=1}^{k-n}(1-q^i).
  \end{align*}
  Thus, we have that
  \begin{multline*}
    \Mal_q^{I_n}(\sigma(-i)=x_i\; \forall\, 0 \leq i \leq k)
    = \P(\sigma(-i)=x_i \; \forall\, 0 \leq i \leq k \mid \sigma_n \text{ perfect}) \\
    =  \mathbbm{1}\!\bigl[x_i \in [-n,0] \hspace{0.5em} \forall\, 0\leq i\leq k\bigr] \P(\sigma(i)=x_i \; \forall\, 1 \leq i \leq k) \frac{\prod_{i=1}^{n-k}(1-q^i)}{\prod_{i=1}^{n+1}(1-q^i)}.
  \end{multline*}
  The ratio of products at the end of the right-hand side tends to one as $n\to\infty$ when $k$ is fixed, and so
  we obtain that \eqref{eq:MallowsdefI} holds as desired.
\end{proof}


\begin{remark}
  The proof of convergence shows that the restriction of $\sigma_n$ to $[-k,0]$ is close to the infinite Mallows permutation not just for fixed $k$ as needed, but even if $k,n\to\infty$ jointly, as long as $n-k\to\infty$.
\end{remark}

\section{Cuts, coupling, and the existence of the tame stable matchings}

In this section we prove \cref{prop:couplingbound}, below, which generalizes \eqref{eq:couplingintro} from the introduction.
We then use this proposition to prove the existential claims from \cref{thm:tamechar}.
We begin by proving a special case of \cref{prop:cutbernoulli}, applying to intervals that are bounded from above. \cref{prop:cutbernoulli} will later follow by an easy limiting argument.
Besides being of independent interest, the proof of \cref{prop:cutbernoullieasy} will serve as a warm-up to the proof of \cref{prop:couplingbound}.
We will then apply \cref{prop:couplingbound} to prove \cref{prop:limits} and \cref{cor:Mallowsaretame}, which establish the existential claims of \cref{thm:tamechar}.
\cref{prop:couplingbound} will also be used in the following section to prove \cref{thm:tameorwild} and complete the proof of \cref{thm:tamechar}.

\begin{prop}
  \label{prop:cutbernoullieasy}
  Let $q\in [0,1)$, let $I$ be an infinite interval that is bounded from above, and let $\sigma$ be a random permutation of $I$ drawn from the Mallows distribution with parameter $q$. Then the process
  \[
    \Bigl(\mathbbm{1}\!\big[i +\tfrac{1}{2}  \text{ is a cut for $\sigma$}\big]\Bigr)_{i \in I}
  \]
  stochastically dominates an i.i.d.\ Bernoulli process with parameter $(1-q)\qPoch$.
\end{prop}

Before beginning the proof of \cref{prop:cutbernoullieasy}, we note that, by re-indexing, \cref{prop:finite} also implies that for every $p\in (0,1]$ and every two non-empty sets $A,B \subseteq \Z$ with $|A|=|B|$, both of which are bounded above, there is a unique stable matching of $K_{A,B}(p)$, which we denote by $\sigma_{A,B}$. Moreover, again by re-indexing and applying \cref{prop:finite}, the stable matching $\sigma_{A,B}$ is perfect almost surely if $A$ is infinite, and with probability $\prod_{i=1}^{|A|}(1-q^i)$ if $A$ is finite.

\begin{proof}
  We assume without loss of generality that $I=(-\infty,-1]$. We define a sequence of partial matchings $\sigma_n : (-\infty,-1] \to [-n,-1] \cup \{-\infty\}$, $n \geq1$, by letting
  \[
    \sigma_n(-i) = \begin{cases}
      \sigma(-i) & \text{ if $i \leq n$ and $-\sigma(-j) \leq n$ for all $ 1\leq j \leq i$}\\
      -\infty & \text{ otherwise}.
    \end{cases}
  \]
  It suffices to prove that
  \[
    \P\left( -n - \tfrac{3}{2} \text{ is a cut for $\sigma$} \mid \sigma_n \right) \geq (1-q)\qPoch.
  \]
  In fact we will prove the equality
  \begin{equation}
    \label{eq:conditional}
    \P\left( -n - \tfrac{3}{2} \text{ is a cut for $\sigma$} \mid \sigma_n \right) = (1-q)\prod_{k=1}^{U_n}(1-q^k),
  \end{equation}
  where $U_n$ is the number of males who are among the $n$ most attractive (i.e., are in $[-n,-1]\times\{\male\}$) and are unmatched in the partial matching $\sigma_n$.
  Note that $-n-\tfrac{3}{2}$ is a cut for $\sigma$ if and only if $U_{n+1}=0$.

 First, the most attractive male queries his compatibility with each of the $n$ most attractive females, i.e., each female in $[-n,-1]\times\{\female\}$.  If he finds he is not compatible with any of them, then we stop the procedure and do not match anyone. Otherwise, he is matched to the most attractive of these females with whom he is compatible. In this case, the second most attractive male queries his compatibility with each of the $n$ most attractive females. If he is not compatible with any of these females other than the one that is already matched, we stop the procedure and do not match any males other than the most attractive one. Otherwise, we match the second most attractive male to the most attractive of these females with whom he is compatible and who is not already matched. We continue this procedure recursively, finding matches for the males in order of attractiveness  until we reach a male who cannot be matched, at which point we stop.
  Let $\cF_n$ be the $\sigma$-algebra generated by all the information concerning compatibility that is revealed when computing $\sigma_n$ via this procedure.

  Now suppose that we wish to compute $\sigma_{n+1}$, given $\sigma_n$ and the $\sigma$-algebra $\cF_n$.
  We know that the most attractive male who is unmatched in $\sigma_n$ is not compatible with any of the females in $[-n,-1]\times\{\female\}$ who are unmatched in $\sigma_n$.
  Other than this, the only information we have about compatibility concerns pairs of males and females at least one of whom is already matched in $\sigma_n$, and this information is no longer relevant for computing $\sigma_{n+1}$.

  Thus, to compute $\sigma_{n+1}$, we use the following procedure, illustrated in \cref{fig:next_man}.
  First, the most attractive male who is unmatched in $\sigma_n$ queries his compatibility with $(-n-1,\female)$, who is always unmatched in $\sigma_n$.
  If he finds he is not compatible with her, we set $\sigma_{n+1}=\sigma_n$ and stop.
  Otherwise, he finds he is compatible with her.
  This occurs with probability $p=1-q$ independently of everything that has happened previously.
  If this is the case, we then try to match the remaining $U_n$ unmatched males with the remaining $U_n$ unmatched females.
  Since no information concerning compatibility between any of these individuals has been revealed, we can re-index and apply  \cref{lem:finite_sets} to deduce that the conditional probability that the stable matching between them is perfect is equal to $\prod_{k=1}^{|U_n|}(1-q^k)$. It follows that
  \begin{equation}
    \P\left( -n - \tfrac{3}{2} \text{ is a cut for $\sigma$} \mid \sigma_n,\, \cF_n \right) = (1-q)\prod_{k=1}^{U_n}(1-q^k),
  \end{equation}
  and the equality \eqref{eq:conditional} follows by taking the conditional expectation over $\cF_n$ given $\sigma_n$.
\end{proof}

  \begin{figure}
    \includegraphics[width=.19\textwidth]{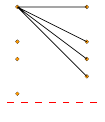}
    \hfill\includegraphics[width=.19\textwidth]{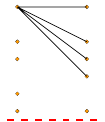}
    \hfill\includegraphics[width=.19\textwidth]{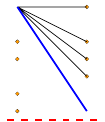}
    \hfill\includegraphics[width=.19\textwidth]{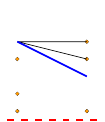}
    \hfill\includegraphics[width=.19\textwidth]{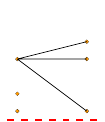}
    \caption{Constructing the matching $\sigma_{n+1}$ from $\sigma_{n}$.
      (i) 4 males and 4 females are unmatched (orange dots). It is known that the most attractive unmatched ``alpha'' male is incompatible (shown as thin black lines) with all unmatched females.
      (ii) A new male and female are revealed.
      (iii) The alpha male is found to be compatible with the new female, so they are matched and removed from consideration.
      (iv) The next most attractive unmatched male is found to be incompatible with the first and second most attractive females but compatible with the third, so they are matched and removed from consideration.
      (v) The next most attractive unmatched male is found to be incompatible with all unmatched females.  He becomes the new alpha male and the step ends.
      Regardless of the initial state, the probability that the step ends with all the individuals matched is at least $(1-q)(q)_\infty$.}
    \vspace{-0.75em}
    \label{fig:next_man}
  \end{figure}

\begin{remark}
  In the sampling algorithm used above, the information concerning compatibilities that is revealed when computing $\sigma_{n+1}$ given $\sigma_n$ and $\cF_n$ is precisely $\cF_{n+1}$.
  In fact, $(U_n)_{n\geq 0}$ is a Markov chain with the filtration $\cF_n$, and \eqref{eq:conditional} gives the probability of jumping to $0$.
  The associated Markov chain is positive recurrent by the above proposition, and a stationary $\Z$-indexed Markov process with the same transition rule can be used to sample the Mallows permutation of $\Z$.
\end{remark}

We now come to the main technical result of this section.
We say that a set $A \subseteq \Z$ is \textbf{low} if it is bounded from above and its complement $\Z\setminus A$ is bounded from below (equivalently, if its symmetric difference with $(-\infty,0]$ is finite).
For each low set $A$, we define $r(A)$ to be the largest integer such that $(-\infty,r] \subseteq A$. In particular, if $\sigma$ is a locally finite matching then the set $\sigma((-\infty,n]) = \{\sigma(i) : i \leq n\}$ is low for each $n\in \Z$, with
\[
  r\bigl(\{\sigma(i) : i \leq n\}\bigr) =  \min_{i>n} \sigma(i) -1.
\]
By re-indexing, it follows from \cref{prop:finite} that for any two low sets $A,B \subseteq \Z$, there is almost surely a unique stable matching of $K_{A,B}(p)$, which we denote $\sigma_{A,B}$, and which is perfect almost surely. Since there are only countably many low sets, this holds for all low sets simultaneously almost surely.
We say that a pair of low sets $A$ and $B$ is \textbf{balanced} if $|A\setminus B|=|B \setminus A|$. Note this condition is equivalent to $|A \cap [-n,\infty)|=|B \cap [-n,\infty)|$ for all sufficiently large $n$, or indeed for all $n \geq r(A)\wedge r(B)$.

\begin{prop}\label{prop:couplingbound}
  Let $p=1-q\in(0,1]$ and consider $K_{\m Z,\m Z}(p)$.
  For $\ell\in\{1,2\}$ let $(A_\ell,B_\ell)$ be a pair of balanced low subsets of $\Z$, set $r_\ell=r(A_\ell)\wedge r(B_\ell)$, and set $\sigma_\ell = \sigma_{A_\ell,B_\ell}$, the unique stable matching of $K_{A_\ell,B_\ell}(p)$.
  Then
  \[
    \P\bigl(\sigma_1(-i) = \sigma_2(-i) \hspace{0.5em} \forall i \geq \max(n -r_1,n-r_2)\bigr) \geq 1 - \Big[ 1-(1-q)\qPoch^2\Big]^{n}
  \]
  for every $n\geq 0$.
  In particular, 
  $\sigma_1(-i) = \sigma_2(-i)$ for all sufficiently large $i$ almost surely.
\end{prop}

We remark that one may deduce a similar result (with a worse constant) from \cref{prop:cutbernoulli} via a finite-energy argument.

Our proof of \cref{prop:couplingbound} will use the following simple correlation inequality.

\begin{lemma}
  \label{lem:finite_sets}
Let $p=1-q\in (0,1]$. If $A,B,A',B' \subseteq \Z$ are finite and non-empty with $|A|=|B|$ and $|A'|=|B'|$, then
  \begin{align*}
    \P(\sigma_{A,B} \text{ and } \sigma_{A',B'} \text{ are both perfect})  &\geq \prod_{k=1}^{|A|} (1-q^k)\prod_{k=1}^{|A'|} (1-q^k).
  \end{align*}
\end{lemma}

In other words, the perfection of the stable matching of different pairs of sets are positively correlated events.
Note that these events are not increasing with respect to the compatibility graph, so that the claimed positive correlation does \emph{not} follow from the FKG inequality.

\begin{proof} We prove the claim by induction on $|A|+|A'|$. The cases in which either $|A|=1$ or $|A'|=1$, and in particular the case $|A|+|A'|=2$, are trivial.
Thus, suppose that $|A|,|A'|\geq 2$, and that the claim has been proven for all pairs of pairs of finite non-empty sets $\tilde A, \tilde B, \tilde A', \tilde B' \subseteq \Z$ with $|\tilde A|=|\tilde B|$,  and $|\tilde A'|=| \tilde B'|$, and $|\tilde A|+|\tilde A'|< |A|+|A'|$. Let $i=\max A \cup A'$, so that $(i,\male)$ is the most attractive male in $(A\cup A')\times\{\male\}$.

  First suppose  that $i$ is in exactly one of $A$ or $A'$; without loss of generality we may assume that $i\in A\setminus A'$.
  As in the proof of \cref{prop:finite},  the male $(i,\male)$ must be matched in $\sigma_{A,B}$ to the most attractive female he is compatible with in $B\times\{\female\}$, and the probability that there is at least one such compatible female is $1-q^{|A|}$. Note that on the event that $(i,\male)$ is matched to $(j,\female)$ in $\sigma_{A,B}$ we have that $\sigma_{A,B}(k)=\sigma_{A\setminus\{i\}, B\setminus\{j\}}(k)$ for every $k\in A\setminus\{i\}$. Note also that the only information required to compute $\sigma_{A,B}(i)$ concerns compatibility information between $(i,\male)$ and $\Z\times\{\female\}$, and that, given $\sigma_{A,B}(i)=j$, this information is no longer relevant for computing either $\sigma_{A\setminus\{i\}, B\setminus\{j\}}$ or $\sigma_{A',B'}$. Thus, it follows by the induction hypothesis that for every $j\in B$,
  \begin{align*}
 \P\left(\sigma_{A,B} \text{ and } \sigma_{A',B'} \text{ are both perfect} \mid \sigma_{A,B}(i) =j \right) &= \P\left(\sigma_{A \setminus \{j\}, B\setminus \{i\}} \text{ and } \sigma_{A',B'} \text{ are both perfect} \right)\\ &\geq \prod_{k=1}^{|A|-1}(1-q^k)\prod_{k=1}^{|A'|} (1-q^k).
  \end{align*}
  The result now follows since $\P(\sigma_{A,B}(i)=j$ for some $j \in B)=1-q^{|A|}$.

Now suppose that $i \in A \cap A'$.
  In order for $\sigma_{A,B}$ and $\sigma_{A',B'}$  both to be perfect,  $(i,\male)$ must be compatible with both a female from $B\times \{\female\}$ and a female from $B'\times \{\female\}$, with these two females possibly being the same.
  The probability of the required females existing is
  \begin{equation}
  \label{eq:qBFKG}
    1 - q^{|B|} - q^{|B'|} + q^{|B\cup B'|} \geq (1-q^{|B|})(1-q^{|B'|}).
  \end{equation}
  Arguing similarly to the previous case, we have that
  \begin{multline*}
\P\left(\sigma_{A,B} \text{ and } \sigma_{A',B'} \text{ are both perfect} \mid \sigma_{A,B}(i) =j, \sigma_{A',B'}(i) =j'   \right) \\= \P\left(\sigma_{A \setminus \{j\}, B\setminus \{i\}} \text{ and } \sigma_{A' \setminus \{i\},B'\setminus \{j'\}} \text{ are both perfect} \right)
  \end{multline*}
  for every $j\in B$ and $j'\in B'$ such that the event being conditioned on has positive probability, and the claim follows from the induction hypothesis together with \eqref{eq:qBFKG}.
\end{proof}

We remark that
  the same argument also yields analogous inequalities for more than two pairs.

\begin{proof}[Proof of \cref{prop:couplingbound}]
  We may assume without loss of generality that $r_1\wedge r_2=0$. The proof is an elaboration of the proof of \cref{prop:cutbernoullieasy}, and we omit some minor details.
  First, observe that if $i\geq 0$ and $-i+\tfrac{1}{2}$ is a cut for both $\sigma_1$ and $\sigma_2$, then the restrictions of $\sigma_1$ and $\sigma_2$ to $(-\infty,-i]$ both define stable matchings of $K_{(-\infty,-i],(-\infty,-i]}(p)$, and hence are equal by \cref{prop:finite}. Thus, it suffices to prove that
  \[
    \P\bigl(\exists i \in [1,n] \text{ such that $-i+\tfrac{1}{2}$ is a cut for both $\sigma_1$ and $\sigma_2$}\bigr)
    \geq 1 -\Big[1-(1-q)\qPoch^2\Big]^n
  \]
  for all $n \geq 0$.

  We perform a similar sampling procedure to that used in the proof of \cref{prop:cutbernoulli}.
  For both $\ell\in \{1,2\}$, we define a sequence of partial matchings $\sigma_{\ell,n} : A_\ell \to B_\ell \cap [-n,\infty) \cup \{-\infty\}$ for $n \geq -1$ by setting $\sigma_{\ell,n}(i)=\sigma_\ell(i)$ if the following three conditions hold for $i\in A_\ell$, and otherwise $\sigma_{\ell,n}(i)=-\infty$.
  \begin{itemize}
    \item $i\geq -n$
    \item $\sigma_\ell(j) \geq -n$ for all $j\in A_\ell \cap [i,\infty)$
    \item $\sigma_{3-\ell}(j) \geq -n$ for all $j\in A_{3-\ell} \cap (i,\infty)$
  \end{itemize}
  It suffices to prove that
  \[
    \P\left( -n - \tfrac{3}{2} \text{ is a cut for both $\sigma_1$ and $\sigma_2$} \mid \sigma_{1,n},\sigma_{2,n} \right) \geq (1-q)\qPoch^2
  \]
  for every $n\geq -1$.  We will prove the stronger bound
  \begin{equation}
    \label{eq:couplingdesired}
    \P\left( - n - \tfrac{3}{2} \text{ is a cut for both $\sigma_1$ and $\sigma_2$} \mid \sigma_{1,n},\sigma_{2,n} \right) \geq (1-q)\prod_{k=1}^{U_{1,n}+1}(1-q^k)\prod_{k=1}^{U_{2,n}+1}(1-q^k)
  \end{equation}
  where $U_{\ell,n}$ is the number of males in $([-n,\infty)\cap A_1)\times \{\male\}$ that are unmatched in $\sigma_{\ell,n}$.
  (Unlike in the proof of Proposition \ref{prop:cutbernoulli}, this is not an equality in general.)


  Similarly to the proof of \cref{prop:cutbernoulli}, we try to match the males in $(A_1\cup A_2)\times \{\male\}$ in order of attractiveness.
  At each step, we may need to find them a match in either $\sigma_{1,n}$, $\sigma_{2,n}$, or both, according to whether they are in $A_1 \times \{\male\}$, $A_2\times\{\male\}$, or both.
  If we reach a male for whom we cannot find both the required matches, we stop.
  If that male can be matched in one of $\sigma_{1,n},\sigma_{2,n}$ but not the other, we make the single match, and stop.
  During this procedure, the information revealed can be described as follows:
  \begin{itemize}
  \item Consider the most attractive male in  $(A_1\cup A_2) \times \{\male\}$ that is unmatched either in $\sigma_{1,n}$, in $\sigma_{2,n}$, or in both.
    For $\ell=1$ or $2$, if this male is unmatched in $\sigma_{\ell,n}$, then we know that he is not compatible with any female in $([-n,\infty) \cap B_\ell)\times \{\female\}$ who is unmatched in $\sigma_{\ell,n}$.
  \item Other than this, the only information we have about compatibility concerns pairs of males and females
    such that for each $\ell=1,2$, at least one of the pair is either already matched in $\sigma_{\ell,n}$ or is not in the set $(A_\ell \times\{\male\}) \cup (B_\ell \times \{\female\})$. The status of these edges is no longer relevant for computing $\sigma_{1,n+1}$ and $\sigma_{2,n+1}$.
  \end{itemize}

  Thus, to compute $\sigma_{1,n+1}$ and $\sigma_{2,n+1}$ given $\sigma_{1,n}$ and $\sigma_{2,n}$, we may do the following.
  First, the most attractive male who is unmatched in either $\sigma_{1,n}$ or $\sigma_{2,n}$ queries his compatibility with $(-n-1,\female)$.
  If he finds he is not compatible with her, we set $\sigma_{1,n+1}=\sigma_{1,n}$ and $\sigma_{2,n+1}=\sigma_{2,n}$ and stop.
  Otherwise, he finds he is compatible with her. This occurs with probability $p=1-q$ independently of everything that has happened previously.
  If this is the case, we match him to her in whichever of the matchings $\sigma_{1,n}$ and/or $\sigma_{2,n}$ that he was previously unmatched in. Call these updated matchings $\sigma'_{1,n}$ and $\sigma'_{2,n}$.

  At this point, the number of remaining unmatched males (and females) in $\sigma'_{\ell,n}$ is either $U_{\ell,n}$ or $U_{\ell,n}+1$, and no information about their compatibility has been revealed so far. Thus, by \cref{lem:finite_sets}, the conditional probability that all the unmatched individuals in $\sigma'_{1,n}$ and in $\sigma'_{2,n}$ both support a perfect stable matching is at least
  \[
    \prod_{k=1}^{U_{1,n}+1} (1-q^k)\prod_{k=1}^{U_{2,n}+1} (1-q^k) \geq \qPoch^2.
  \]
  Thus the conditional probability that $-n-\frac32$ is a cut for both $\sigma_1,\sigma_2$ is at least $(1-q)\qPoch^2$, and the claim follows.

  The final part of the proposition follows by continuity of measure.
\end{proof}

\begin{cor}[Existence of limit matchings]
  \label{prop:limits}
  Let $p\in (0,1]$.
  Let $n \in \Z$, and for each $m\geq0$ let $\sigma_{n,m}$ be the unique stable matching of $K_{(-\infty,m+n],(-\infty,m]}(p)$.
  Then $\sigma_{m,n}$ converges almost surely as $m\to\infty$ to a permutation $\sigma_n$ of $\Z$, and the limit is distributed as the composition of the Mallows permutation of $\Z$ with the shift $\theta_n: i \mapsto i+n$.
\end{cor}

\begin{proof}
  It suffices to consider the case $n=0$, as the others then follow by re-indexing.
  We write $\sigma_m=\sigma_{0,m}$.
  For each $k\in\Z$, it follows from \cref{prop:couplingbound} that
  \[
    \sum_{m>k} \P\Bigl(\sigma_{m}(k) \neq \sigma_{m+1}(k) \Bigr) < \infty.
  \]
  It follows by Borel-Cantelli that $\sigma_{m}$ almost surely converges pointwise to a function $\sigma : \Z \to \Z$.
  By symmetry, the inverse functions $\sigma^{-1}_m$ also almost surely converge pointwise to a function $\sigma^{-1}:\Z\to\Z$.
  Since  $\sigma^{-1}_m \circ \sigma_m(i) = \sigma_m \circ \sigma^{-1}_m(i) = i$ for every $i \in (-\infty,m]$, it follows that $\sigma^{-1}(\sigma(i))=\sigma(\sigma^{-1}(i))=i$ for every $i\in \Z$, so that $\sigma$ is a permutation almost surely.
  The fact that $\sigma$ is Mallows-distributed with parameter $q$, and in particular that the limit defining this permutation exists, follows from the corresponding statement for $\sigma_m$, proven in \cref{prop:finite}.
\end{proof}

\begin{proof}[Proof of \cref{prop:cutbernoulli}]
  This follows from \cref{prop:cutbernoullieasy} and \cref{prop:limits}, by taking the result through to the limit.
\end{proof}

We call the matching $\sigma_n$ from \cref{prop:limits} the \textbf{Mallows matching} of $K_{\Z,\Z}(p)$ with flow $n$.

\begin{cor}[Mallows matchings are tame]
  \label{cor:Mallowsaretame}
  Let $q \in [0,1)$, and let $\sigma$ be a random permutation of $\m Z$ drawn from the Mallows distribution with parameter $q$. Then almost surely the composition of $\sigma$ with the shift $\theta_n: i \mapsto i+n$ is tame for all $n$.
\end{cor}

\begin{proof}
  Tameness is clearly invariant to composition with a shift, so it suffices to consider the case $n=0$.
  Observe that if $(T_i + \tfrac{1}{2})_{i\in \Z}$ are the sequence of cuts of $\sigma$ in order, then for every $i\in\Z$
  \[
    \max\left\{M(\sigma, j+\tfrac{1}{2}): T_i \leq j \leq T_{i+1}\right\} \leq \left|T_{i+1}-T_i\right|.
  \]
  Thus, if $\sigma$ is the Mallows permutation of $\Z$, then it follows from \cref{prop:cutbernoulli} that
  \[
    \P\left(M(\sigma, i + \tfrac{1}{2}) \geq 2m \right) \leq c_1(q)^m
  \]
  for some $c_1(q)<1$.
  It follows by Borel-Cantelli that $M(\sigma,i+\tfrac{1}{2}) = O(\log |i|)$ almost surely, so that $\sigma$ is tame as claimed.
\end{proof}

We also have the following immediate corollaries of \cref{prop:couplingbound}, showing that every perfect locally finite matching is ``tame towards $-\infty$''.

\begin{cor}
  \label{cor:eventualcoupling2}
  Let $p\in (0,1]$ and consider $K_{\Z,\Z}(p)$. Almost surely, every perfect, locally finite, balanced stable matching $\sigma$ of $K_{\Z,\Z}(p)$ has a cut.
\end{cor}

\begin{cor}
  \label{cor:eventualcoupling}
  Let $p\in (0,1]$ and consider $K_{\Z,\Z}(p)$. Then the following holds almost surely. For every perfect, locally finite stable matching $\sigma$ of $K_{\Z,\Z}(p)$ with flow $n$, there exists $N_\sigma$ such that $\sigma(-i)=\sigma_n(-i)$ for every $i \geq N_\sigma$, where $\sigma_n$ is the Mallows matching of $K_{\Z,\Z}(p)$ with flow $n$.
\end{cor}

Finally, let us note that \cref{thm:tameorwild} is sharp in the sense that, disregarding constants, the definition of tameness cannot be strengthened.

\begin{prop}
  \label{prop:tamesharp}
  Let $p=1-q\in (0,1]$ and let $\sigma$ be the balanced Mallows matching of $K_{\Z,\Z}(p)$. Then
  \[\liminf_{i\to\infty} \frac{M\left(\sigma,i+\tfrac{1}{2}\right)}{\log |i|} \geq \frac{1}{2\log q^{-1}}\]
  almost surely.
\end{prop}

\begin{proof}
  For each $i\in \Z$, let $X_i = \min\{ |i-j| : (j,\female) \text{ is compatible with } (i,\male) \}$.
  Then the random variables $X_i$ are i.i.d.\ with $\P(X_i \geq n) = q^{2n-1}$ for all $n \geq 1$.
  Thus, for all $a>0$ and for all $i\geq e^{1/a}$ we have that
  \[
    \P(X_i \geq a \log i) = q^{2\lfloor a\log i\rfloor-1}.
  \]
  When $a^{-1}=2\log q^{-1}$, the latter expression has infinite sum (over $i\geq e^{1/a}$) and thus by the Borel-Cantelli Lemma $X_i \geq a\log |i|$ infinitely often, almost surely.
  The claim follows since $M(\sigma,i+\frac12)\wedge M(\sigma,i-\frac12) \geq X_i$.
\end{proof}

\section{The dichotomy between tame and wild}

In this section we complete the proof of \cref{thm:tameorwild,thm:tamechar}. The central additional ingredients required are the following two lemmas concerned with perfect, locally finite stable matchings.

\begin{lemma}
  \label{lem:tamesufficient}
  Let $p\in (0,1]$, let $q=1-p$, and consider the random bipartite graph $K_{\Z,\Z}(p)$. There exists a positive constant $c_{2}(q)$ such that the following holds almost surely: Every perfect, locally finite stable matching $\sigma$ of $K_{\Z,\Z}(p)$ satisfying
  \begin{equation}
    \liminf_{i\to+\infty}\frac{1}{i} M\left(\sigma,i+\tfrac{1}{2}\right) \leq c_2(q)
    \label{eq:smalllinear}
  \end{equation}
  is equal to the Mallows stable matching of $K_{\Z,\Z}(p)$ with the same flow as $\sigma$. In particular, $\sigma$ is tame if and only if \eqref{eq:smalllinear} holds.
\end{lemma}

\begin{lemma}
  \label{lem:wildsufficient}
  Let $p\in (0,1]$, let $q=1-p$, and consider the random bipartite graph $K_{\Z,\Z}(p)$.
  Almost surely, every perfect, locally finite stable matching $\sigma$ of $K_{\Z,\Z}(p)$ satisfying
  \begin{equation}
  \label{eq:wildsufficientlinear}
    \liminf_{i\to+\infty}\frac{1}{i}M(\sigma,i+\tfrac{1}{2}) > 0
  \end{equation}
  is wild. In particular, $\sigma$ is wild if and only if \eqref{eq:wildsufficientlinear} holds.
\end{lemma}

Before proving these lemmas, let us use them to prove \cref{thm:tameorwild} and \cref{thm:tamechar}.

\begin{proof}[Proof of \cref{thm:tameorwild}]
Let $\Omega$ be the almost sure event that for every $i\in \Z$, the sets $\{j : (j,\female)$ compatible with $(i,\male)\}$ and $\{j : (j,\male)$ compatible with $(i,\female)\}$ are both unbounded below. We claim that every locally finite stable matching of $K_{\Z,\Z}(p)$ is perfect on the event $\Omega$. Indeed, suppose without loss of generality that $(i,\male)$ is unmatched in some stable matching $\sigma$ of $K_{\Z,\Z}(p)$. Then each element of the set $\{j : (j,\female)$ compatible with $(i,\male)\}$ must be matched to a male more attractive than $(i,\male)$, and thus $\sigma$ is not locally finite. The remaining claims of the theorem follow from \cref{lem:tamesufficient,lem:wildsufficient}, since at least one of \eqref{eq:smalllinear} or \eqref{eq:wildsufficientlinear} must hold for every $\sigma$.
\end{proof}

\begin{proof}[Proof of \cref{thm:tamechar}]
The fact that a tame stable matching of flow $n$ exists almost surely for each $n$ follows from follows from \cref{prop:limits,cor:Mallowsaretame}, while the fact that the matching is unique and can be described as the limit of the unique stable matchings $\sigma_{n,m}$ of $K_{(-\infty,m],(-\infty,m+n]}(p)$ as $m\to\infty$ follows from \cref{lem:tamesufficient}. Given this limiting construction, \cref{thm:tamechar}.iii follows from \cref{prop:finite}.
  For \cref{thm:tamechar}.iv, fix $m \geq 0$ and $n \in \Z$. Let $(i_k)_{k\geq0}$ and $(j_k)_{k\geq1}$ be decreasing sequences defined recursively by $i_0= m+n+1$, $j_k=\sigma_{m,n+1}^{-1}(i_{k-1})$ for every $k\geq 0$ and $i_k = \sigma_{m,n}(j_k)$ for every $k\geq 1$.
  Considering the gender-reversal of the recursive procedure for computing $\sigma_{m,n}$ and $\sigma_{m,n+1}$ as in the proof of \cref{prop:finite}, beginning by assigning matches to the most attractive females in $(-\infty,m+n]$ and $(-\infty,m+n+1]$ respectively, it follows by a straightforward induction argument that
  \begin{align*}\sigma_{m,n+1}^{-1}(i) &= \begin{cases} j_{k+1} & \text{ if $i=i_k$ for some }k \geq 0,\, \\
      \sigma_{m,n}^{-1}(i) &  \text{otherwise} \end{cases}
    \intertext{and similarly that}
    \sigma_{m,n}(j) &= \begin{cases} i_k & \text{ if $j=j_k$ for some $k\geq 1$}\\
      \sigma_{m,n+1}(i) & \text{otherwise.} \end{cases}
  \end{align*}
  The second part of \cref{thm:tamechar}.iv follows by taking the limit as $m\to\infty$.
  The first part follows from the second, since if for each $n$ males do no worse in $\sigma_{n+1}$ than in $\sigma_n$, then the same comparison holds for any $\sigma_n,\sigma_{n'}$.
\end{proof}

\begin{proof}[Proof of \cref{lem:tamesufficient}]
  By re-indexing, it suffices to consider the balanced case of flow $0$.
  Let $0<a<1$, and let $\sA_{i,a}$ be the set of low subsets $A$ of $\Z$ such that the pair of low sets $((-\infty,i],A)$ is balanced, and the symmetric difference of $A$ with $(-\infty,i]$ is contained in $[(1-a)i,(1+a)i]$.
  Clearly $|\sA_{i,a}|\leq 2^{2ai+1}$. Let $k\geq 0$. By Proposition~\ref{prop:couplingbound}, for every $A\in \sA_{i,a}$ we have that
  \[
    \P\Bigl(\sigma_{(-\infty,i],A}(j)=\sigma_{(-\infty,i],(-\infty,i]}(j) \text{ for every } j \leq k\Bigr) \geq 1-c_1(q)^{(1-a)i-k},
  \]
  where $c_1(q)=1-(1-q)(q)_\infty^2$. Thus, it follows by the union bound that
  \[
    \P\left(\sigma_{(-\infty,i],A}(j)=\sigma_{(-\infty,i],(-\infty,i]}(j)  \text{ for every } j \leq k \text{ and $A\in \sA_{i,a}$}\right) \geq 1- 2^{2ai+1}c_1(q)^{(1-a)i-k}.
  \]

  Fix $a>0$ such that $2^{2a}c_1(q)^{(1-a)}<1$ (such an $a$ exists since $c_1(q)<1$). By the Borel-Cantelli lemma and the previous lower bound, the event
  \[
    \Omega_a:= \forall k \ \exists \ell \ \forall i\geq \ell \
    \left\{ \sigma_{(-\infty,i],A}(j)=\sigma_{(-\infty,i],(-\infty,i]}(j)  \text{ for every } j \leq k \text{ and $A\in \sA_{i,a}$}\right\}
  \]
  occurs almost surely. Consider the set of balanced, perfect, locally finite stable matchings $\sigma$ of $K_{\Z,\Z}(p)$ that satisfy
  \[
    \liminf_{i\to+\infty}\frac{M\left(\sigma,i+\tfrac{1}{2}\right)}{i} < a.
  \]
  For any such matching, $\sigma((-\infty,i])$ belongs to $\sA_{i,a}$ for infinitely many $i$.
  However, on the event $\Omega_a$, the balanced Mallows stable matching of $K_{\Z,\Z}(p)$ is the only stable matching of $K_{\Z,\Z}(p)$ with this property, and so $\sigma$ must be equal to this matching.
\end{proof}

\begin{proof}[Proof of \cref{lem:wildsufficient}]
  By re-indexing it suffices to consider balanced stable matchings. Moreover, since by \cref{cor:eventualcoupling2} there is a probability one event on which every perfect, locally finite, balanced stable matching of $K_{\Z,\Z}(p)$ has a cut, it suffices by re-indexing to consider matchings that have a cut at $\tfrac{1}{2}$. That is, it suffices to prove that there is a probability one event on which
  every perfect, locally finite, balanced stable matching $\sigma$ of $K_{\Z,\Z}(p)$ that has a cut at $1/2$ and satisfies
  \begin{equation}
    \liminf_{i\to+\infty}\frac{1}{i}M(\sigma,i+\tfrac{1}{2}) > 0
  \end{equation}
  is wild.

  For $\eps,\delta>0$, define $\Omega_{\eps,\delta}$ to be the event on which there exists $N<\infty$ such that the following conditions hold for all $i\geq N$.

  \begin{samepage}
  \begin{enumerate}
  \item[(i)]
    The inequalities \[\#\bigl\{i \leq j \leq  i + \eps i : (j,\male) \text{ is compatible with } (i+k,\female) \bigr\} \geq \tfrac{1-q}{2}\eps i\]
and
\[\#\bigl\{i \leq j \leq  i + \eps i : (j,\female) \text{ is compatible with } (i+k,\male) \bigr\} \geq \tfrac{1-q}{2}\eps i \] both
     hold for every $ 1\leq k \leq e^{\delta i}$.
  \item[(ii)]
    There do not exist sets $A,B \subseteq [0,e^{\delta i}]$ with $|A|,|B| \geq \frac{1-q}{2} \eps i$ such that no male in $A \times \{\male\}$ is compatible with any female in $B \times \{\female\}$.
  \end{enumerate}
\end{samepage}

  We claim that for each $\eps>0$, there exists $\delta(\eps,q)>0$ such that $\Omega_{\eps,\delta}$ occurs almost surely.
Indeed, for condition (i), the Chernoff bound implies that there exists a constant $c_{3}(\eps,q)<1$ such that for each $i,k\geq 0$, the probability that either of the sets in question is smaller than $(1-q)\eps i/2$ is at most $c_3(\eps,q)^i$. Thus, summing over the possible choices of $k$ and $i$ and applying Borel-Cantelli shows that if $\delta$ is sufficiently small then the required inequalities hold for all $i$ sufficiently large almost surely.

  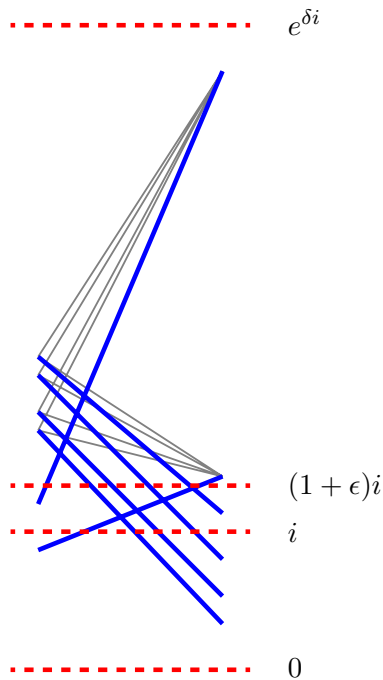
\begin{figure}
    \includegraphics[height=.4\textheight]{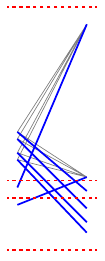}
    \setlength{\unitlength}{.4\textheight}
    \begin{picture}(0,1)
      \put(0,.03){$0$}
      \put(0,.225){$i$}
      \put(0,.295){$(1+\epsilon)i$}
      \put(0,.96){$e^{\delta i}$}
    \end{picture}
    \caption{Exponential blow-up in wild matchings.  Edges whose length is linear in their height necessitate edges whose length is exponential in their height (with high probability as the height goes to infinity).}
  \end{figure}

  For condition (ii), it suffices to consider $A,B$ of the minimal possible size $s = \frac{1-q}{2} \eps i$.
  Counting the choices of $A$ and $B$ and using the the union bound gives the following upper bound on the probability that there exist sets $A,B$ violating (ii):
  \[
    \P\bigl(\text{there exist such }A,B \subseteq [0,e^{\delta i}]\bigr)
    \leq \binom{e^{\delta i}}{s}^2 q^{s^2}
    \leq e^{2\delta i s - s^2 |\log q|} = e^{C(\eps,\delta)i^2},
  \]
  where we have used the elementary inequality $\binom{m}{n}\leq m^n$.
  If $\delta$ is sufficiently small then $C(\eps,\delta)<0$, and the result follows by Borel-Cantelli.
  This completes the proof of the claim.


Now suppose that the almost sure event $\Omega = \bigcap_{n\geq 1} \Omega_{1/n,\delta(1/n)}$ occurs.
 Let $\sigma$ be a stable perfect matching of $K_{\Z,\Z}(p)$ that is balanced and has a cut at $\tfrac{1}{2}$, let $n\geq 2$ be such that
   $\eps =1/n < \frac12 \liminf M(\sigma,i+1/2)/i$, and let $\delta=\delta(\eps,q)$.
  For each $a,b\geq 0$, we define the indicator functions
  \[
    F(a,b) = \mathbbm{1}\!\left(\exists i \leq a \text{ such that $\sigma(i)>b$ or $\sigma^{-1}(i)>b$}\right).
  \]
  We claim that if $i\geq N$ and $F(i,i+\eps i)=1$, then we must have that $F(i+\eps i,e^{\delta i})=1$ also.
  Given this, it is easily seen that $\sigma$ is wild: Indeed, our choice of $\eps$ guarantees that for every sufficiently large $i$ there is an edge of length at least $2\eps i$ spanning $i$, so that at least one of $F(i,i+\eps i)$ or $F(i-\eps i,i)$ is equal to $1$ for every sufficiently large $i$.
  In the first case we have that $F(i+\eps i, e^{\delta i})=1$ and in the second we have that $F(i,e^{\delta(i-\eps i)})=1$. In either case we deduce that $M(\sigma,2i+1/2) \geq e^{\delta i/2}$,  concluding the proof.

  To prove the claim above, let $k=\lceil \eps i \rceil$, and suppose that $F(i,i+k)=1$.
  By symmetry between males and females, let us suppose without loss of generality that some $j\leq i$ has $\sigma(j) > i+k$.
  If $\sigma(j)>e^{\delta i}$, we are done.
  Otherwise, let $V$ be the set of males in $(i,i+k]\times \{\male\}$ that are compatible with $\sigma(j)$.
  By (i) we have that $|V| \geq \frac{1-q}{2} k$.
  Note that each male $(a,\male)\in V$ must have $\sigma(a) > \sigma(j)$, since otherwise $\sigma$ would be unstable.
  Moreover, if any $a\in V$ has $\sigma(a)>e^{\delta i}$ we are done, so we may suppose not.

  Since $\sigma$ is balanced and has a cut at $1/2$, there is a set $W\subseteq [1,i+k]$ of size at least $k$ such that $\sigma^{-1}(w) > i$ for every $w\in W$.
  If any $w\in W$ has $\sigma^{-1}(w) > e^{\delta i}$, we are done, so suppose not.
  Since each female in $W\times\{\female\}$ is less attractive than $\sigma(j)$ and is matched to a male more attractive than $(j,\male)$, stability implies that they cannot be matched to any male in $V\times\{\male\}$.
  It follows that there is a set $W'\subset W$ of females of size at least $|V|$ that are matched to males in $(i+k,e^{\delta i}]\times\{\male\}$.

  Let $B=\sigma(V)$ and $A=\sigma^{-1}(W')$.
  Summarizing our conclusions, $V,W'$ are sets of size at least $\frac{1-q}{2} k$ in $[0,i+k]$, and $A,B$ are sets of size at least $\frac{1-q}{2} k$ in $(i+k,e^{\delta i}]$.
  Stability implies that no male in $A\times\{\male\}$ is compatible with any female in $B\times\{\female\}$, which contradicts (ii).
\end{proof}

\begin{remark}
  It is possible to show using the above proof that for every $p=1-q \in (0,1)$, there exists a positive constant $c_3(q)$ such that
  \[
    \liminf_{i\to+\infty}\frac{\log M\bigl(\sigma,i+\tfrac{1}{2}\bigr)}{i} \geq c_3(q)
  \]
  for every wild stable matching of $K_{\Z,\Z}(p)$ almost surely.
\end{remark}

\section{Existence of wild stable matchings}

\begin{proof}[Proof of \cref{thm:wildexistence}]
  For a finite set $A\subset\Z$ and for an integer $i\in\Z\setminus A$, let $\Omega_{A,i}$ be the event that both of the sets
  \[
    \bigl\{j\in\m Z\colon (j,\female)\text{ is compatible with $(i,\male)$ but not $(k,\male)$ for any $k\in A$}\bigr\}
  \]
  and
  \[
    \bigl\{j\in\m Z\colon (j,\male)\text{ is compatible with $(i,\female)$ but not $(k,\female)$ for any $k\in A$}\bigr\}
  \]
  are unbounded from above.
  Clearly $\P(\Omega_{A,i})=1$ for all $A$ and $i$.
  Since there are only a countable number of finite subsets of $\Z$, the event
  $$
  \Omega=\bigcap \{\Omega_{A,i}\colon A\text{ finite subset of }\m Z,\ i\in \mathbb Z\setminus A\}
  $$
  also has probability $1$.

  We will show that the conclusions of the theorem hold on $\Omega$.
  All three claims will follow from variants of the same construction.
  Let $(x_n)_{n\geq 0}$ be an enumeration of a subset of $\Z \times \{\male,\female\}$ (representing individuals to be matched in a given order), and let $(a_j)_{j\geq 1}$ be a sequence of positive integers (which will encode uncountably many options).
  We take the enumerated subset to be one of the following:
  \begin{enumerate}
  \item $\big(\Z \times \{\male\}\big) \cup \big([0,\infty)\times \{\female\}\big)$
  \item $\Z\times \{\male,\female\}$.
  \item $[0,\infty) \times \{\male,\female\}$.
  \end{enumerate}
  These three choices are used to construct matchings that are respectively: not perfect; perfect but not locally finite; and perfect, locally finite, and wild.
  For each fixed sequence $(x_n)_{n\geq0}$, every choice of the sequence $(a_j)_{j\geq 1}$ will yield a different stable matching with this desired properties.

  At each step of the construction, we choose the next individual in the sequence, and if they are not already matched, we find a compatible partner for them who is incompatible with everyone previously matched.
  At the end of step $n\geq 0$ of the construction, we have a partial matching $\sigma_n$ with the property that the vertices $x_1,\dots,x_n$ are all matched, and that $\sigma_n$ is a stable perfect matching of the subgraph induced by the set of all matched vertices.

  Initially we take $\sigma_0$ to be the empty matching.
  Given $\sigma_{n}$ for some $n\geq 0$, we define $\sigma_{n+1}$  as follows.
  \begin{itemize}
  \item If $x_{n+1}$ is already matched in $\sigma_n$ then we set $\sigma_{n+1}=\sigma_{n}$. (That is, we do nothing.)
  \item If $x_{n+1}$ is not matched in $\sigma_n$, then we choose a partner for it as follows.
    Suppose $x_{n+1}$ is a female (the case of a male $x_{n+1}$ is similar).
    We will match $(x_{n+1},\female)$ to a male in the set $S \times \{\male\}$, where
    \[
      S = \left\{k \geq 0 \colon \begin{array}{l} (k,\male)\text{ is more attractive than any male that has already been}\\ \text{matched, is compatible with $x_{n+1}$, and is not compatible with any}\\ \text{female that is matched in $\sigma_n$}\end{array}\right\}.
    \]
    (We take $k\geq 0$ so that $S$ is bounded from below.) On $\Omega$, the set $S$ is unbounded from above.
    If $\sigma_n$ has $j$ edges, we match $(x_{n+1},\female)$ to the $a_{j+1}$-th least attractive male in $S\times\{\male\}$.
  \end{itemize}
\noindent
  Let $\sigma = \lim \sigma_n$. As a result, the following hold.
  \begin{itemize}
  \item If $(x_i)_{i\geq 0}$ is an enumeration of $A = (\Z \times \{\male\}) \cup ([0,\infty)\times\{\female\})$, then $\sigma$ is a perfect matching of $A$, with $(-\infty,-1]\times\{\female\}$ left unmatched.
    Such a matching must also be a stable matching of $K_{\Z,\Z}(p)$, since every male in the matching prefers their partner to every female in $(-\infty,-1]$.
  \item If $(x_i)_{i\geq 0}$ is an enumeration of $\Z\times\{\male,\female\}$, taking the limit as $n \to\infty$ we obtain a perfect stable matching of $K_{\Z,\Z}(p)$. The construction ensures that for every $k>0$, the partner of $(-k,\female)$ belongs to $[0,\infty)\times \{\male\}$, so that $\sigma$ is not locally finite.
  \item If $(x_i)_{i\geq 0}$ is an enumeration of $[0,\infty)\times \{\male,\female\}$, then $\sigma$ is a stable perfect matching of $K_{[0,\infty),[0,\infty)}(p)$.
    Combining this with the unique stable matching of $K_{(-\infty,-1],(-\infty,-1]}(p)$, we obtain a stable perfect matching of $K_{\Z,\Z}(p)$ with a cut at $-\tfrac{1}{2}$, which must be balanced.
  \end{itemize}

  It is clear that, in each case, different sequences $(a_n)$ yield different matchings $\sigma$.
  Indeed, at the first place two sequences differ, the edge added to the matchings will also differ.
  In the last case (locally finite matchings) we have not ruled out that $\sigma$ is the unique  balanced, tame stable matching. However, excluding this matching still leaves uncountably many locally finite, balanced, wild stable matchings.
  Similar statements for other values of the flow follow by re-indexing.

Finally, note that in the third case, the restrictions to $K_{[0,\infty),[0,\infty)}(p)$ of the matchings that we obtain  are stable and perfect, and by re-indexing we obtain that $K_{I,I}(p)$ has uncountably many wild stable matchings almost surely for every infinite interval $I$ that is bounded from below.
\end{proof}

Finally, we show that \cref{thm:tameorwild} is sharp in the sense that, disregarding constants, the definition of wildness cannot be strengthened.

\begin{prop}
  \label{prop:wildsharp}
  Let $p=1-q\in (0,1)$ and consider the random bipartite graph $K_{\Z,\Z}(p)$.
  Then there almost surely exists a perfect, locally finite, wild stable matching of $K_{\Z,\Z}(p)$ such that
  \begin{equation}
    \label{eq:wildsharpdesired}
    \limsup_{i\to+\infty} \frac{\log M\bigl(\sigma,i+\tfrac{1}{2}\bigr)}{i} \leq 2 \log q^{-1}.
  \end{equation}
\end{prop}

\begin{proof}
  Let $(x_i)_{i\geq 1}$ be the enumeration of $[1,\infty) \times \{\male,\female\}$ given by $x_{2i-1}=(i,\male)$ and $x_{2i}=(i,\female)$ for every $i\geq 1$.
  As in the proof of \cref{thm:wildexistence}, we define a sequence of partial matchings $(\sigma_n)_{n\geq0}$ as follows.
  Let $\sigma_{0}$ be the unique stable perfect matching of $K_{(-\infty,0],(-\infty,0]}$.
  Having defined $\sigma_{n}$, if $x_{n+1}$ is matched in $\sigma_{n}$ then $\sigma_{n+1} = \sigma_{n}$.
  If $x_{n+1}$ is not matched in $\sigma_{n}$,  let $\sigma_{n+1}$ be obtained from $\sigma_n$ by matching $x_n$ to the least attractive member of the opposite gender that is compatible with $x_{n+1}$,  is strictly more attractive than every individual (of either gender) that is already matched in $\sigma_n$, and is incompatible with every individual of the same gender as $x_{n+1}$ that has positive index and is already matched in $\sigma_n$.
  As before, we obtain a locally finite stable perfect matching $\sigma$ of $K_{\Z,\Z}(p)$ by taking the limit as $n\to\infty$.

  We wish to verify that this matching $\sigma$ satisfies \eqref{eq:wildsharpdesired}. For each $n\geq 0$, let $H_n$ be the index of the most attractive individual (of either gender) that is already matched in $\sigma_n$. For each $n\geq 1$, let $N_n$ be the number of individuals of positive index that are of the same gender as $x_n$ and already matched in $\sigma_{n-1}$, and note that $\lfloor (n-1) /2\rfloor \leq N_n \leq n-1$. Then, conditional on $\sigma_n$, we either have that $x_{n+1}$ is already matched in $\sigma_n$ in which case  $H_{n+1}-H_n=0$, or else $H_{n+1}-H_n-1$ is conditionally distributed as a geometric random variable with success probability $pq^n \leq p q^{N_{n+1}} \leq p q^{\lfloor n/2\rfloor}$. In particular, it follows that $H_{n+1}-H_{n}\geq p^{-1}q^{-\lfloor n/2\rfloor}$ for infinitely many $n$ almost surely, which clearly implies that $\sigma$ is not tame and is therefore wild almost surely by \cref{thm:tameorwild}. On the other hand, it also follows that
\[\E[H_n]\leq \sum_{i=0}^{n-1} \Bigl(1+\frac{1}{pq^{i}}\Bigr)=q^{-n + o(n)},\]
and it follows by an easy application of Markov's inequality and Borel-Cantelli that
\[\limsup_{n\to\infty} \frac{1}{n}\log H_n \leq \log q^{-1}\]
almost surely. This immediately implies the claim. 
\end{proof}

\subsection*{Acknowledgments}

This work was carried out while OA was visiting and TH was an intern at Microsoft Research, Redmond.
OA is supported in part by NSERC.
TH was also supported by a Microsoft Research PhD Fellowship.

\addcontentsline{toc}{chapter}{Bibliography}

\bibliographystyle{habbrv}
\bibliography{matching}

\end{document}